\documentclass{article}
\usepackage{graphicx} 
\usepackage[left=0.8in, right=0.8in, top=1.5in, bottom=1.5in]{geometry}

\usepackage[T1]{fontenc} 
\usepackage[utf8]{inputenc} 
\usepackage[english]{babel}
\usepackage{amssymb}
\usepackage{lipsum} 
\usepackage{lmodern}
\usepackage{amsmath}
\usepackage{amsthm}
\usepackage{bm}
\usepackage{mathtools}
\usepackage{braket}
\usepackage{appendix}
\usepackage{esint}
\newcommand{\abs}[1]{{\left|#1\right|}}

\usepackage{enumitem}
\usepackage{booktabs}
\usepackage{graphicx}
\usepackage{tikz}
\usetikzlibrary{patterns}
\usepackage{multicol}
\usepackage{caption}
\usepackage{bigints}
\usepackage[skins,theorems]{tcolorbox}
\tcbset{highlight math style={enhanced,
colframe=black,colback=white,arc=0pt,boxrule=1pt}}
\def\XXint#1#2#3{{\setbox0=\hbox{$#1{#2#3}{\int}$}
    \vcenter{\hbox{$#2#3$}}\kern-.5\wd0}}

\theoremstyle{definition}
\newtheorem{definizione}{Definition}[section]
\theoremstyle{plain}
\newtheorem{teorema}{Theorem}[section]

\newtheorem{lemma}[teorema]{Lemma}

\theoremstyle{definition}
\newtheorem{esempio}{Example}[section]
\newtheorem{oss}[esempio]{Remark}

\newtheorem*{open*}{Open problems}
\DeclareMathOperator{\R}{\mathbb{R}}

\makeatletter
\newcommand{\myfootnote}[2]{\begingroup
	\def\@makefnmark{}%
	\addtocounter{footnote}{-1}%
	\footnote{\textbf{#1} #2}
	\endgroup}
\makeatother

\usepackage{hyperref}

\title{Sharp lower bound for the Monge-Ampère torsion on convex sets}
\author{Francesco Salerno}
\date{}

\newcommand{\Addresses}{{
\bigskip 
   
   \medskip 
     \textit{E-mail address}, F.~Salerno: \texttt{f.salerno@ssmeridionale.it} 
   \medskip

 \textsc{Mathematical and Physical Sciences for Advanced Materials and Technologies, Scuola Superiore Meridionale, Largo San Marcellino 10, 80138 Napoli, Italy.}

 \par\nopagebreak 

}} 

\begin{document}
\maketitle
\begin{abstract}
    The \emph{Monge-Ampère} torsion deficit of an open, bounded convex set $\Omega\subset\R^n$ of class $C^2$ is the normalized gap between the value of the torsion functional evaluated on $\Omega$ and its value on the ball with the same $(n-1)$-quermassintegral as $\Omega$. Using the technique of the \emph{shape derivative}, we prove that the ratio between this deficit and to a geometric deficit arising from the \emph{Alexandrov-Fenchel inequality}, for any given family of open, bounded convex sets of $\R^n$ ($n\geq2$) of class $C^2$, smoothly converging to a ball, is bounded from below by a dimensional constant. We also show that this ratio is always bounded from above by a constant.
    \newline
    \newline
    \textsc{Keywords:}
    Monge-Ampère operator, torsional rigidity, shape derivative.
    \newline
    \textsc{MSC 2020:} 
    35B35, 35J96, 52A40.
\end{abstract}

\section{Introduction}
Given an open convex set $\Omega\subset\mathbb{R}^n$ of class $C^2$, starting from the solution of the following problem
\begin{equation}\label{original prob}
    \begin{cases}
        \det D^2u_\Omega=1 & \text{ in } \Omega\\
        u_\Omega=0 & \text{ on } \partial\Omega
    \end{cases},
\end{equation}
where $D^2u$ is the Hessian matrix of $u$, in \cite{Brandolini2009NewIE} (subsequently proven in \cite{BNT2} by the same authors using different techniques) it is proved that in the class of open, bounded convex sets of class $C^2$ with prescribed volume, the ellipsoid minimizes the torsion functional, namely
\begin{equation}
    \label{min tor}T(\Omega)=\|u_\Omega\|_{L^1(\Omega)}^n\geq \|u_{\mathcal{E}(\Omega)}\|_{L^1(\mathcal{E}(\Omega))}^n=T(\mathcal{E}(\Omega))=\frac{|\Omega|^{n+2}}{(n+2)^n\omega_n^2}.
\end{equation}
Here $|\cdot|$ denotes the Lebesgue measure of a set and $\omega_n$ is the measure of the unit ball (see \cite{L} for the extension of \eqref{min tor} to open, bounded convex sets not necessarily $C^2$).
In \cite{Talenti1981}, in dimension $n=2$, then extended in \cite{tso} to $n\geq2$, prescribing another geometric quantity, the $(n-1)$-th quermassintegral of $\Omega$, denoted by $W_{n-1}(\Omega)$ (see \S \ref{Sezione 2} for the precise definition), it was proved that the torsion functional is maximized by a ball, namely if $\Omega^*_{n-1}$ is the ball with the same $(n-1)$-th quermassintegral as $\Omega$, then
\begin{equation}\label{max tor}
    T(\Omega)\leq T(\Omega^*_{n-1})=\frac{\omega_n^n\zeta_{n-1}^{n(n+2)}(\Omega)}{(n+2)^n},
\end{equation}
where $\zeta_{n-1}(\Omega)$ is the radius of the ball $\Omega^*_{n-1}$.
From \eqref{max tor} it follows that the quantity
\begin{equation*}
    \delta T(\Omega)=\frac{T(\Omega^*_{n-1})-T(\Omega)}{T(\Omega^*_{n-1})}=1-\frac{T(\Omega)}{T(\Omega^*_{n-1})}
\end{equation*}
is always nonnegative. On the other hand, by the Alexandrov-Fenchel inequality (see \S \ref{Sezione 2} for a precise statement) the quantity
\begin{equation*}
    \delta_{AF}(\Omega)=1-\omega_n^{(n-1)(n+2)}\frac{W_0^{n+2}(\Omega)}{W_{n-1}^{n(n+2)}(\Omega)}
\end{equation*}
is also always nonnegative, and it is zero if and only if $\Omega$ is a ball (see, for instance, \cite{Schneider}). In this work, we want to study the ratio $\delta T/\delta_{AF}$. In more details we prove that this ratio is bounded from above by a constant, and this follows directly from \eqref{min tor}-\eqref{max tor}.
In order to estimate the ratio from below, we exploit the technique of shape derivative to investigate the behavior of $\delta T/\delta_{AF}$ along an arbitrary family of open, bounded convex sets that converges in a suitable way to a ball (see Definition \ref{convergence} below). With this technique in \cite{Nitsch2012AnIR} the ratio $\delta P/\delta\lambda$, that is, the deficit of the isoperimetric inequality divided by the deficit of the Faber-Krahn inequality, has been studied. In the following, we will use the following definition.
\begin{definizione}\label{convergence}
    We say that a family $\{\Omega(t)\}_t$ of open, bounded convex sets of $\mathbb{R}^n$ smoothly converges to an open, bounded convex set $\Omega$ as $t$ goes to zero, if there exists a positive number $\delta$ and a one-parameter family of transformations $\Phi_t:\mathbb{R}^n\rightarrow\mathbb{R}^n$, $t\in[0,\delta)$, such that
    \begin{itemize}
        \item [(i)] $\Phi_t$, $\Phi^{-1}_t\in C^\infty(\mathbb{R}^n,\mathbb{R}^n)$ for all $t\in[0,\delta)$,
        \item [(ii)] the maps $t\mapsto \Phi_t(x)$, $t\mapsto \Phi^{-1}_t(x)$ belong to $C^\infty([0,\delta))$ for all $x\in\mathbb{R}^n$,
        \item [(iii)] $\Omega(0)=\Omega$ and $\Omega(t)=\Phi_t(\Omega)$ for all $t\in[0,\delta)$.
    \end{itemize}
\end{definizione}
We now state the main theorem.
\begin{teorema}\label{MainTheo}
    For any $n\geq 2$, there exists a dimensional constant $c_n$ such that, for any one-parameter family of open, bounded convex sets $\Omega(t)\subset\mathbb{R}^n$ of class $C^2$, with fixed $(n-1)$-quermassintegral, smoothly converging to a ball as $t\rightarrow0^+$, then
    \begin{equation*}
        c_n\leq\liminf_{t\rightarrow0^+}\frac{\delta T(\Omega(t))}{\delta_{AF}(\Omega(t))},
    \end{equation*}
    and the constant is 
    \begin{equation*}
        c_n=\frac{n-1}{n}.
    \end{equation*}
    Moreover, for all open, bounded convex sets of class $C^2$ the following inequality holds
    \begin{equation*}
        \frac{\delta T(\Omega)}{\delta_{AF}(\Omega)}\leq1.
    \end{equation*}
\end{teorema}
The paper is organized as follows. In section \ref{Sezione 2} we fix the notations and we recall all the results that we need in the proof. In section \ref{sezione 3} we perform all the computations that we need in the proof of the main theorem, in more details we write the condition that comes out by fixing the $(n-1)$-th quermassintegral and we write in terms of spherical harmonics the two deficits $\delta T$ and $\delta_{AF}$. In section \ref{Sezione 4} we prove Theorem \ref{MainTheo} and, finally, in the appendix we show how to treat the term that comes out from the nonlinearity of the Monge-Ampère operator.

\section{Tools and notations}\label{Sezione 2}
\subsection*{The generalized Kronecker delta and the $k$-th elementary symmetric polynomial}
In the following we use Einstein's notation for repeated indices. We recall that the \emph{generalized Kronecker delta} is defined as
\begin{equation*}
    \delta_{j_1\dots j_k}^{i_1\dots i_k}=\begin{cases}
        1 & \text{ if $i_1,\dots,i_k$ are distinct integers and are an even permutation of $j_1,\dots,j_k$},\\
        -1 & \text{ if $i_1,\dots,i_k$ are distinct integers and are an odd permutation of $j_1,\dots,j_k$},\\
        0 & \text{ in all other cases}.
    \end{cases}
\end{equation*}
The latter can be also seen in terms of an $n\times n$ determinant as follows
\begin{equation*}
    \delta_{j_1\dots j_k}^{i_1\dots i_k}=\begin{vmatrix}
            \delta_{j_1}^{i_1} & \delta_{j_2}^{i_1} & \dots & \delta_{j_k}^{i_1}\\
            \delta_{j_1}^{i_2} & \delta_{j_2}^{i_2} & \dots & \delta_{j_k}^{i_2}\\
            \vdots & \vdots & \ddots & \vdots\\
            \delta_{j_1}^{i_k} & \delta_{j_2}^{i_k} & \dots & \delta_{j_k}^{i_k}
\end{vmatrix}.
\end{equation*}
We mention a contraction property of the generalized Kronecker delta which will be useful in the following. Let $0<p<k$ be two integers, then
\begin{equation}\label{contraction}
    \delta_{j_1\dots j_p}^{i_1\dots i_p}\delta_{j_1\dots j_k}^{i_1\dots i_k}=\frac{p!(n-k+p)!}{(n-k)!}\delta_{j_{p+1}\dots j_k}^{i_{p+1}\dots i_k}.
\end{equation}
In particular, if $p=1$, formula \eqref{contraction} reduces to
\begin{equation}\label{contr2}
    \delta_{j_l}^{i_l}\delta_{j_1\dots j_k}^{i_1\dots i_k}=(n-k+1)\delta_{j_1\dots j_{l-1}j_{l+1}\dots j_k}^{i_1\dots i_{l-1}i_{l+1}\dots i_k},
\end{equation}
where $l\in\{1,\dots,k\}.$
\begin{definizione}
    Consider $\lambda=(\lambda_1,\dots,\lambda_n)\in\R^n$. We define the \emph{$k$-th elementary symmetric polynomial} of $\lambda$ as follows
\begin{equation*}
    \begin{split}
        S_k(\lambda)&=
        \sum_{1\leq i_1<\dots<i_k\leq n}\lambda_{i_1}\dots \lambda_{i_k}, \quad 1\leq k\leq n,\\
        S_0(\lambda)&=1.
    \end{split}
\end{equation*}
\end{definizione}
If $\lambda_1,\dots,\lambda_n$ are the eigenvalues of a matrix $A={A^i_j}$, we denote $S_k(A)=S_k(\lambda)$ and the $k$-th elementary symmetric polynomial can be defined equivalently from the elements of $A$ as follows
\begin{equation}\label{symmPol}
    S_k(A)=\frac{1}{k!}\delta^{j_1\dots j_k}_{i_1\dots i_k}A^{i_1}_{j_1}\dots A^{i_k}_{j_k},
\end{equation}
where we are using the Einstein convention to sum over repeated indices. If $k=1$, $S_1(A)$ is the trace of $A$. On the other hand, if $k=n$, $S_n(A)$ is the determinant of $A$.
\begin{definizione}
    Let $A_1,\dots, A_k$ be $n\times n$ matrices. We define the \emph{Newton transformation tensor} $[T_k]_i^j$ as
\begin{equation*}
    [T_k]_i^j(A_1,\dots,A_k)=\frac{1}{k!}\delta_{ii_1\dots i_k}^{jj_1\dots j_k}A_{j_1}^{i_1}\dots A_{j_k}^{i_k}.
\end{equation*}
\end{definizione}
If $A=A_1=\dots=A_k$, we will simplify the notation as $[T_k]_i^j(A)=[T_k]_i^j(A,\dots,A)$. The Newton transformation tensor allows us to write the \emph{cofactor matrix} of $S_k$. Indeed, the last is the matrix $\{S_k^{ij}(A)\}_{ij}$ whose elements are defined by
\begin{equation*}
    S_k^{ij}(A)=\frac{\partial}{\partial A_j^i}S_k(A).
\end{equation*}
Then, from \eqref{symmPol}, we have
\begin{equation}\label{cofproof}
        S_k^{ij}(A)=\frac{\partial}{\partial A_j^i}\frac{1}{k!}\delta_{ii_1\dots i_{k-1}}^{jj_1\dots j_{k-1}}A_j^iA_{j_1}^{i_1}\dots A_{j_{k-1}}^{i_{k-1}}=\frac{1}{k!}\delta_{ii_1\dots i_{k-1}}^{jj_1\dots j_{k-1}}A_{j_1}^{i_1}\dots A_{j_{k-1}}^{i_{k-1}}=\frac{1}{k}[T_{k-1}]_i^j(A).
\end{equation}

\subsection*{Quermassintegrals and the Alexandrov-Fenchel deficit}  \label{quermass}
 For the content of this section, we will refer to \cite{HCG, Schneider}. Let $\Omega \subset \R^n$ be a non-empty, bounded, convex set, let $B_1$ be the unit ball centered at the origin and $\rho > 0$. We can write the Steiner formula for the Minkowski sum $\Omega+ \rho B_1$ as
\begin{equation}
    \label{Steiner_formula}
    \abs{\Omega + \rho B_1} = \sum_{i=0}^n \binom{n}{i} W_i(\Omega) \rho^{i} .
\end{equation}
The coefficients $W_i(\Omega)$ are known in the literature as quermassintegrals of $\Omega$. In particular, $W_0(\Omega) = \abs{\Omega}$,  $nW_1(\Omega) = P(\Omega)$ and $W_n(\Omega) = \omega_n$ where $\omega_n$ is the measure of $B_1$.

Formula \eqref{Steiner_formula} can be generalized to every quermassintegral, obtaining
\begin{equation} \label{Steinerquermass}
    W_j(\Omega+\rho B_1) = \sum_{i=0}^{n-j} \binom{n-j}{i} W_{j+i}(\Omega) \rho^i, \qquad j=0, \ldots, n-1.
\end{equation}

If $\Omega$ has $C^2$ boundary, the quermassintegrals can be written in terms of principal curvatures of $\Omega$. More precisely, denoting with $\sigma_k$ the $k$-th normalized elementary symmetric function of the principal curvature $\kappa_1, \ldots, \kappa_{n-1}$ of $\partial \Omega$, i.e.
 \begin{equation}\label{meancurv}
 \sigma_0 = 1, \qquad \qquad \sigma_j = \binom{n-1}{j}^{-1} \sum_{1 \leq i_1 < \ldots < i_j \leq n-1} \kappa_{i_1} \ldots \kappa_{i_j}, \qquad j = 1,\ldots,n-1,
\end{equation}
 then, the quermassintegrals can be written as
 \begin{equation}
     \label{quermass_con_curvature}
     W_j(\Omega) = \frac{1}{n} \int_{\partial \Omega} \sigma_{j-1} \, d \mathcal{H}^{n-1}, \qquad j = 1,\ldots, n-1.
 \end{equation}
 We recall that $\sigma_j$ in \eqref{meancurv} is also known as $j$-th mean curvature of $\Omega$, while $\sigma_1$ is called just mean curvature. In the following, we will denote the $(n-1)$-th mean curvature, also known as \emph{Gaussian curvature}, by $k_{\partial\Omega}$.
 We remark that the $j$-th mean curvature is, up to a constant, the $j$-th elementary symmetric polynomial applied to the \emph{second fundamental form} of $\partial\Omega$.
 \begin{definizione}
    Let $\Omega$ be a convex set, we define the \textit{j-th mean radius}, $j=1,\dots,n-1$, of $\Omega$ as
\begin{equation}
    \label{meanradii}
    \zeta_j(\Omega)=\left(\frac{W_j(\Omega)}{\omega_n}\right)^\frac{1}{n-j}.
\end{equation}
\end{definizione}
\begin{oss}
    The $j$-th mean radius $\zeta_j(\Omega)$ of a convex set $\Omega$ is the radius of a ball with same $j$-th quermassintegral as $\Omega$.
\end{oss}

Furthermore, Aleksandrov-Fenchel inequalities hold true
\begin{equation}
    \label{Aleksandrov_Fenchel_inequalities}
    \biggl( \frac{W_j(\Omega)}{\omega_n} \biggr)^{\frac{1}{n-j}} \geq \biggl( \frac{W_i(\Omega)}{\omega_n} \biggr)^{\frac{1}{n-i}}, \qquad 0 \leq i < j \leq n-1,
\end{equation}
where equality holds if and only if $\Omega$ is a ball. When $i=0$ and $j=1$, formula \eqref{Aleksandrov_Fenchel_inequalities} reduces to the classical isoperimetric inequality, i.e.
\[
    P(\Omega) \geq n \omega_n^{\frac{1}{n}} \abs{\Omega}^{\frac{n-1}{n}}.
\]
If we choose $i=0$, $j=n-1$ in \eqref{Aleksandrov_Fenchel_inequalities}, then we have
\begin{equation*}
    W_{n-1}(\Omega)\geq \omega_n^\frac{n-1}{n}W_0(\Omega)^\frac{1}{n}.
\end{equation*}
\begin{definizione}
    We define the \emph{Alexandrov-Fenchel deficit} as
\begin{equation*}
    \delta_{AF}=1-\omega_n^{(n-1)(n+2)}\frac{W_0^{n+2}(\Omega)}{W_{n-1}^{n(n+2)}(\Omega)}.
\end{equation*}
\end{definizione}

\subsection*{The Monge-Ampère operator and the Torsion functional}
 For the content of this section, we refer to \cite{Lions, tso, Wang}. We want to define the Monge-Ampère operator by framing it within a broader class of differential operators: the $k$-Hessian operators. Let $\Omega$ be an open subset of $\R^n$ and let $u\in C^2(\Omega)$. The $k$-Hessian operator is the $k$-th elementary symmetric function of the Hessian matrix $D^2u$, i.e., it is defined as

\begin{equation}\label{kess}
    S_k(D^2 u)= \sum_{1\le i_1< \dots < i_k\le n} \lambda_{i_1}\cdots\lambda_{i_k}, \quad k=1,\dots, n,
\end{equation}
where $\lambda_i$ are the eigenvalues of $D^2u$. We note that $S_k$ is a second-order differential operator and it reduces to the Monge-Ampère operator for $k=n$ and to the Laplace operator for $k=1$. Except for the case $k=1$, these operators are fully nonlinear and non-elliptic, unless one restricts to the class  of $k$-convex functions (see for instance \cite{caffarelli})

$$\mathcal{A}_k(\Omega)= \left\{u\in C^2(\Omega) : S_i (D^2u)\ge 0 \, \text{ in } \, \Omega, \,i=1,\dots, k\right\}.$$

The operator $S_k^{\frac{1}{k}}$ is homogeneous of degree 1, and if we denote by

$$S_k^{ij}(D^2u)=\frac{\partial}{\partial u_{ij}}S_k(D^2u),$$
the Euler identity for homogeneous functions gives
\begin{equation*}
    S_k(D^2u)=\frac{1}{k} S_k^{ij}(D^2u)u_{ij}.
\end{equation*}
A direct computation shows that the $\left(S_k^{1j}(D^2u), \dots, S_k^{nj}(D^2u)\right)$ is divergence-free, hence $S_k(D^2u)$ can be written in divergence form

\begin{equation}
    \label{skdiv}
    S_k(D^2u)=\frac{1}{k}\left(S_k^{ij}(D^2u) u_j\right)_i,
\end{equation}
where the subscripts $i, j$ stand for partial differentiation. 

If $u\in C^2(\Omega)$ and $t$ is a regular point of $u$, on the boundary of $\{u\le t\}$ it is possible to link the $k$-Hessian operator with the $(k-1)$-th mean curvature (see \cite{reilly73, trudinger97})

\begin{equation}
    \label{hksk}
   \binom{n-1}{k-1} \sigma_{k-1}=\frac{S_k^{ij}(D^2u)u_iu_j}{\abs{\nabla u}^{k+1}}.
\end{equation}
The previous identity for $k=n$ is the following
\begin{equation}
    \label{curv}k_{\partial\Omega}=\frac{S_n^{ij}(D^2u)u_iu_j}{\abs{\nabla u}^{n+1}}.
\end{equation}
Given $\Omega$ an open, convex set, we can define the $k$-Torsional rigidity as
    \begin{equation}
        \label{TorsDef}T_k(\Omega)=\max\left\{\frac{\left(\int_\Omega -v\,dx \right)^{k+1}}{\int_\Omega-vS_k(D^2v)\,dx}\,:\, \text{$v\in \mathcal{A}_k(\Omega)$, $v|_{\partial\Omega}=0$ }  \right\}
    \end{equation}
    The maximum in \eqref{TorsDef} is achieved by the solution to
    \begin{equation}
    \label{torsDirich}
        \begin{cases}
            S_k(D^2u_\Omega)= \binom{n}{k}& \text{ in $\Omega$, }\\
            u_\Omega=0 & \text{ on $\partial \Omega$}.
        \end{cases} 
    \end{equation}
    We note that, from the equation in \eqref{torsDirich}, we have
    \begin{equation*}
        T_k(\Omega)=\|u_\Omega\|_{L^1(\Omega)}^k.
    \end{equation*}
    Another consequence of the result in \cite{tso} is that the $k$-Torsional rigidity is maximum on the ball in the class of open, convex sets with prescribed $(k-1)$-th quermassintegral. The latter's proof is based on a symmetrization technique, introduced in \cite{Talenti1981} for the Monge-Ampère equation in the two-dimensional case, then extended in \cite{tso} for the $k$-Hessian equation in any dimension, called quermassintegral-symmetrization (see also \cite{BT} for more details about this symmetrization, \cite{DPGC} for the extension to the anisotropic setting, \cite{MS} for some quantitative improvement).
When $\Omega=B_1$, the solution of \eqref{torsDirich} is given by
\begin{equation}\label{Sol on ball}
    u_{B_1}(x)=\frac{|x|^2-1}{2}.
\end{equation}
We note that, since we are interested in convex functions, here and in what follows $u<0$ in $\Omega$. Throughout the remainder of the paper we restrict to the case $k=n$. For notational convenience, we denote $T_n(\cdot)$ simply by $T(\cdot)$.
\section{Computations of the Shape Derivatives}\label{sezione 3}
In this section we will perform all the computations that we need to prove the main result.

\subsection*{First and Second Variation of $W_{n-1}$}
We remark that, since we are fixing the $(n-1)$-quermassintegral, it holds that
\begin{equation}\label{condizioni}
    \frac{d}{dt}W_{n-1}(\Omega(t))=\frac{d^2}{dt^2}W_{n-1}(\Omega(t))=0.
\end{equation}
To make the previous two conditions explicit, we want to parametrize the $k$-th mean curvature of $\Omega(t)$. From \cite{VanBlargan2022QuantitativeQI} we have the following lemma
\begin{lemma}
Let $\Omega\subset\mathbb{R}^n$. If we can parametrize the boundary of $\Omega$ with a function $r\in C^2(\partial B_1)$, then we have
    \begin{equation*}
        \sigma_k=\binom{n-1}{k}^{-1}(r^2+|\nabla_\xi r|^2)^{-\frac{k+2}{2}}\sum_{m=0}^k(-1)^m\binom{n-1-m}{k-m}r^{-m}\left(r^2S_m(D^2r)+\frac{n-1+k-2m}{n-1-m}r_ir_j[T_m]_i^j(D^2r) \right).
    \end{equation*}
\end{lemma}
\begin{teorema}
    Let $\{\Omega(t)\}_t$ be a family of open, bounded convex sets of $\R^n$, with $\Omega(0)=B_1$. Then we can write the $(n-1)$-th quermassintegral as follows
    \begin{equation*}
        \begin{split}
            W_{n-1}(\Omega(t))&=\omega_n+\frac{t}{n}\int_{\partial B_1}V(\xi)\,d\mathcal{H}^{n-1}_\xi\\
            &+\frac{t^2}{2n(n-1)}\int_{\partial B_1}\left[(n-1)A(\xi)+2(n-3)S_2(D^2V)-(n^2-6n+7)|\nabla_\xi V|^2\right]\,d\mathcal{H}^{n-1}_\xi,
        \end{split}
    \end{equation*}
    where we have represented the boundary $\partial\Omega(t)$ as
    \begin{equation*}
        r(\xi,t)=1+tV(\xi)+\frac{t^2}{2}A(\xi)+o(t^2),
    \end{equation*}
    with $o(t^2)$ uniform in $\xi$.
\end{teorema}
\begin{proof}
From the previous lemma, we have
\begin{equation}\label{PaolCor}
    \begin{split}
        W_{n-1}(\Omega(t))&=\frac{1}{n(n-1)}\int_{\partial\Omega(t)}\sigma_{n-2}\,d\mathcal{H}^{n-1}\\
        &=\frac{1}{n(n-1)}\int_{\partial B_1}\sum_{m=0}^{n-2}\frac{(-1)^m(n-1-m)r^{n-m-2}}{(r^2+|\nabla_\xi r|^2)^\frac{n-1}{2}}\left(r^2S_m(D^2r)+\frac{2n-2m-3}{n-1-m}r_ir_j[T_m]_i^j(D^2r) \right)\,d\mathcal{H}_\xi^{n-1}.
    \end{split}
\end{equation}
We observe that, when we represent the boundary $\partial\Omega(t)$ as 
\begin{equation*}
    r(\xi,t)=1+tV(\xi)+\frac{t^2}{2}A(\xi)+o(t^2),
\end{equation*}
the second term in the sum in \eqref{PaolCor} is of higher order in $t$, except when $m=0$ that gives us
\begin{equation*}
    (2n-3)r^{n-2}\delta_i^j r_i r_j=  (2n-3)r^{n-2}|\nabla_\xi r|^2.
\end{equation*}
In the first term of the sum in \eqref{PaolCor} we have
\begin{equation*}
    \begin{split}
        &S_0(D^2r)=1,\quad S_1(D^2r)=t\Delta_\xi V+\frac{t^2}{2}\Delta_\xi A+o(t^2),\\
        &S_2(D^2r)=t^2S_2(D^2V)+o(t^2),\\
        &S_m(D^2r)=o(t^2),\quad \forall\,m\ge 3.
    \end{split}
\end{equation*}
Then we just have to consider
\begin{equation}\label{1}
    \frac{(2n-3)r^{n-2}|\nabla_\xi r|^2+\sum_{m=0}^{2}(-1)^m(n-1-m)r^{n-m}S_m(D^2r)}{(r^2+|\nabla_\xi r|^2)^\frac{n-1}{2}}=:\frac{f}{g}.
\end{equation}

Then we can expand
\begin{equation}\label{2}
    \begin{split}
         &f=(2n-3)r^{n-2}|\nabla_\xi r|^2+(n-1)r^nS_0(D^2r)-(n-2)r^{n-1}S_1(D^2r)+(n-3)r^{n-2}S_2(D^2r)+o(t^2)\\
         &=n-1+t[n(n-1)V-(n-2)\Delta_\xi V]\\[1.5ex]
         &+\frac{t^2}{2}\left[n(n-1)^2V^2+n(n-1)A-2(n-1)(n-2)V\Delta_\xi V-(n-2)\Delta_\xi A+2(n-3)S_2(D^2V)+(4n-6)|\nabla_\xi V|^2\right]\\[1.5ex]
         &+o(t^2).
    \end{split} 
\end{equation}
Now we expand $g$. We start with
\begin{equation*}
    (r^2+|\nabla_\xi r|^2)^\frac{1}{2}=1+tV+\frac{t^2}{2}(A+|\nabla_\xi V|^2)+o(t^2),
\end{equation*}
and then
\begin{equation}\label{3}
    \begin{split}
        g&=(r^2+|\nabla_\xi r|^2)^\frac{n-1}{2}=\left\{1+\left[tV+\frac{t^2}{2}(A+|\nabla_\xi V|^2) \right] \right\}^{n-1}+o(t^2)\\
        &=\sum_{j=0}^{n-1}\binom{n-1}{j}\left[tV+\frac{t^2}{2}(A+|\nabla_\xi V|^2) \right]^j+o(t^2)\\
        &=1+t(n-1)V+\frac{t^2}{2}[(n-1)(A+|\nabla_\xi V|^2)+(n-1)(n- 2 )V^2]+o(t^2).
    \end{split}
\end{equation}
Finally, joining \eqref{PaolCor}-\eqref{1}-\eqref{2}-\eqref{3}, we get
\begin{equation}
\label{espansione W_{n-1}}
    \begin{split}
        &W_{n-1}(\Omega(t))=\omega_n+\frac{t}{n(n-1)}\int_{\partial B_1}[(n-1)V-(n-2)\Delta_\xi V]\,d\mathcal{H}^{n-1}_\xi\\
        &+\frac{t^2}{2n(n-1)}\int_{\partial B_1}\left[(n-1)A-(n-2)\Delta_\xi A+2(n-3)S_2(D^2V)-(n^2-6n+7)|\nabla_\xi V|^2 \right]\,d\mathcal{H}^{n-1}_\xi+o(t^2).
    \end{split}
\end{equation}
Now, applying \emph{the Green-Beltrami identity} (\cite{AH}), we have
\begin{equation*}
    \int_{\partial B_1}\Delta_\xi V\,d\mathcal{H}_\xi^{n-1}=\int_{\partial B_1}\Delta_\xi A\,d\mathcal{H}_\xi^{n-1}=0,
\end{equation*}
and the claim follows.
\end{proof}
\begin{oss}
    From the previous theorem, if we fix $W_{n-1}(\Omega(t))=\omega_n$ for all $t\geq0$, then the two conditions \eqref{condizioni} are
    \begin{equation}
    \label{der sec}
        \begin{split}
            &\int_{\partial B_1}V\,d\mathcal{H}^{n-1}_\xi=0,\\
            &\int_{\partial B_1}\left[(n-1)A-(n^2-6n+7)|\nabla_\xi V|^2+2(n-3)S_2(D^2V) \right]\,d\mathcal{H}^{n-1}_\xi=0.
        \end{split}
    \end{equation}
\end{oss}

\subsection*{Shape derivative of the torsion functional and the Alexandrov-Fenchel deficit}
\label{Torsion Deficit}
We start this section by considering the torsion problem for the Monge-Ampère operator in $\Omega(t)$ and the problems related to the time derivatives $u_t$, $u_{tt}$. If we write the operator in divergence form, as explained in \eqref{skdiv}, then we are considering the solution to
\begin{equation}\label{probtor}
    \begin{cases}
        \frac{1}{n}(S_n^{ij}(D^2u)u_j)_i=1 & \text{ in } \Omega(t)\\
         u=0 & \text{ on } \partial\Omega(t)
    \end{cases}.
\end{equation}
Then the problem derived with respect to the time parameter is the following (see also \cite{BNT2})
\begin{equation}\label{probder}
    \begin{cases}
        (S_n^{ij}(D^2u)u_{tj})_i=0 & \text{ in } \Omega(t)\\
        u_t+\nabla u\cdot V=0 & \text{ on } \partial\Omega(t)
    \end{cases},
\end{equation}
where $V=\frac{\partial\Phi_t}{\partial t}$ is the velocity field, and the second derivative of the problem is given by
\begin{equation}\label{probdersec}
    \begin{cases}
        S_n^{ij}(D^2u)u_{ttij}+(S_n^{ij}(D^2u))_tu_{tij}=0 & \text{ in } \Omega(t)\\
        u_{tt}+V^T\cdot D^2 u\cdot V+\nabla u \cdot V_t+2\nabla u_t\cdot V =0 & \text{ on } \partial\Omega(t)
    \end{cases},
\end{equation}
where the superscript $T$ denotes the transpose of a matrix and $V_t=\frac{\partial^2\Phi_t}{\partial t^2}$ is the accelleration field.

\begin{oss}\label{RemarkVel}
    From the boundary condition in \eqref{probder} we have that
    \begin{equation*}
       V=-\frac{u_t}{\abs{\nabla u}}\nu,
    \end{equation*}
    where $\nu$ is the outer normal vector.
\end{oss}
Let us now define $j(t)=T(\Omega(t))$. We now aim to compute the first and second derivative of $j(t)$ at $t=0$. This is the content of the next result.

\begin{teorema}\label{3.3}
    Let $\Phi_t$ be a family of $C^2$ deformations such that $\Phi_0$ is the identity and $W_{n-1}(B_1)=W_{n-1}(\Phi_t(B_1))$ for all $t\geq0$. Then the first and second shape derivatives of the torsion functional with respect to this deformation are given by
    \begin{equation*}
        \begin{split}
            j'(0)&=\frac{\omega_n^{n-1}}{(n+2)^{n-1}}\int_{\partial B_1}u_t \,d\mathcal{H}^{n-1},\\
            j''(0)&=\frac{n-1}{n}\omega_n^{n-2}\left(\int_{\partial B_1}u_t\,d\mathcal{H}^{n-1} \right)^2+\frac{\omega_n^{n-1}}{(n+2)^{n-1}}\\
            &\times\left\{\int_{\partial B_1}\left[u_t^2+V_t|_{t=0}\cdot \nu-2u_t\frac{\partial u_t}{\partial \nu}\right]\,d\mathcal{H}^{n-1}+\int_{B_1}(S_n^{ij}(D^2u))_tu_{tij}u\,dx +n\int_{\partial B_1}(V|_{t=0}\cdot \nu)^2\,d\mathcal{H}^{n-1} \right\}.
        \end{split}
    \end{equation*}
\end{teorema}
\begin{proof}
Let $\Omega(t)=\Phi_t(B_1)$. We want to compute the first and second shape derivative of $T(\Omega(t))$. We recall the well-known \emph{Hadamard's formula} \cite{HP}
\begin{equation}\label{Hadamard}
    \frac{d}{dt}\int_{\Omega(t)}f\,dx=\int_{\Omega(t)} [f_t+\mathrm{div}(fV)]\,dx.
\end{equation}
We recall that $j(t)$ is defined as follows
\begin{equation*}
    j(t)=T(\Omega(t))=\left(\int_{\Omega(t)}-u\,dx \right)^n,
\end{equation*}
where $u$ is the solution to \eqref{probtor}. Then, from \eqref{skdiv}, \eqref{curv}, \eqref{probder} and \eqref{Hadamard}, since $u=0$ on $\partial\Omega(t)$, we have
\begin{equation*}
    \begin{split}
    j'(t)&=n\left(\int_{\Omega(t)}-u\,dx\right)^{n-1}\int_{\Omega(t)}-u_t\,dx=\left(\int_{\Omega(t)}-u\,dx\right)^{n-1}
    \int_{\Omega(t)}-u_t(S_n^{ij}(D^2u)u_j)_i\,dx\\
    &=\left(\int_{\Omega(t)}-u\,dx\right)^{n-1}
    \int_{\partial\Omega(t)}\frac{S_n^{ij}(D^2u)u_ju_i}{\abs{\nabla u}}V\cdot\nu\,d\mathcal{H}^{n-1}+\left(\int_{\Omega(t)}-u\,dx\right)^{n-1}
    \int_{\Omega(t)}-u_t(S_n^{ij}(D^2u)u_j)_i\,dx\\
    &=\left(\int_{\Omega(t)}-u\,dx\right)^{n-1}\int_{\partial\Omega(t)}k_{\partial\Omega(t)}|\nabla u|^{n+1}V\cdot \nu\, d\mathcal{H}^{n-1},
    \end{split}
\end{equation*}
where the last equality follows from an integration by parts. Iterating \eqref{Hadamard} we can write the second derivative as follows
\begin{equation*}
    \begin{split}
        j''(t)&=\frac{n-1}{n}\left(\int_{\Omega(t)}-u\,dx\right)^{n-2}\left(\int_{\partial\Omega(t)}k_{\partial\Omega(t)}|\nabla u|^{n+1}V\cdot \nu\, d\mathcal{H}^{n-1}\right)^2\\
        &+n\left(\int_{\Omega(t)}-u\,dx\right)^{n-1}\int_{\Omega(t)}[-u_{tt}-\mathrm{div}(u_tV)]\,dx.
    \end{split}
\end{equation*}
On the other hand, from \eqref{skdiv}, \eqref{curv} and \eqref{probdersec}, we have that
\begin{align*}
    \int_{\Omega(t)}-u_{tt}\,dx&=\frac{1}{n}\int_{\Omega(t)}-u_{tt}(S_n^{ij}(D^2u)u_j)_i\,dx=\frac{1}{n}\int_{\partial\Omega(t)}-u_{tt}k_{\partial\Omega(t)}|\nabla u|^n\,d\mathcal{H}^{n-1}-\frac{1}{n}\int_{\Omega(t)}S_n^{ij}(D^2u)u_{ttij}u\,dx\\
    &=\frac{1}{n}\int_{\partial\Omega(t)}k_{\partial\Omega(t)}|\nabla u|^n(V^T\cdot D^2 u\cdot V+\nabla u \cdot V_t+2\nabla u_t\cdot V)\,d\mathcal{H}^{n-1}+\frac{1}{n}\int_{\Omega(t)}(S_n^{ij}(D^2u))_tu_{tij}u\,dx , \\
    \int_{\Omega(t)}-\mathrm{div}&(u_tV)\,dx  =-\int_{\partial\Omega(t)}u_tV\cdot \nu\,d\mathcal{H}^{n-1}=\int_{\partial\Omega(t)}|\nabla u|(V\cdot \nu)^2\,d\mathcal{H}^{n-1}.
\end{align*}
Then the second shape derivative of the torsion is
\begin{equation*}
    \begin{split}
        &j''(t)=\frac{n-1}{n}\left(\int_{\Omega(t)}-u\,dx\right)^{n-2}\left(\int_{\partial\Omega(t)}k_{\partial\Omega(t)}|\nabla u|^{n+1}V\cdot \nu\, d\mathcal{H}^{n-1}\right)^2\\
        &+ \left(\int_{\Omega(t)}-u\,dx\right)^{n-1}\left[\int_{\partial\Omega(t)}k_{\partial\Omega(t)}|\nabla u|^n (V^T\cdot D^2 u\cdot V+\nabla u \cdot V_t+2\nabla u_t\cdot V)\,d\mathcal{H}^{n-1}+\underbrace{\int_{\Omega(t)}(S_n^{ij}(D^2u))_tu_{tij}u\,dx}_{NT}\right]\\
        &+n\left(\int_{\Omega(t)}-u\,dx\right)^{n-1}\int_{\partial\Omega(t)}|\nabla u|(V\cdot \nu)^2\,d\mathcal{H}^{n-1},
    \end{split}
\end{equation*}
where "NT" is the effect of the nonlinearity of the operator.

From Remark \ref{RemarkVel} we have that  
\begin{equation*}
    V|_{t=0}=-\frac{u_t}{|\nabla u|}\nu
\end{equation*}
where $u_t$ is the solution of
\begin{equation}
\label{omega}
    \begin{cases}
        (S_n^{ij}(D^2u)u_{ti})_j=0 & \text{ in } B_1\\
        u_t=-|\nabla u|V|_{t=0}\cdot \nu & \text{ on } \partial B_1,
    \end{cases}
\end{equation}
where, since both $\delta T$ and $\delta_{AF}$ are invariant with respect to homotheties, we have made the normalization \\$W_{n-1}(\Omega(t))=\omega_n$ for all $t\geq0$, that is, $\Omega(t)_{n-1}^*=B_1$ for all $t\geq0$.
We observe that, when $t=0$,  we have
\begin{itemize}
    \item $\nabla u\cdot V_t|_{t=0}=|\nabla u|\nu\cdot V_t|_{t=0}$,
    \item $2\nabla u_t\cdot V|_{t=0}=-\frac{2}{|\nabla u|}u_t\frac{\partial u_t}{\partial \nu}$,
    \item $k_{\partial\Omega(t)}|_{t=0}=\frac{1}{\zeta_{n-1}^{n-1}(\Omega(0))}=1$,
\end{itemize}
and 
\begin{equation*}
    \begin{split}
        V^T\cdot D^2u\cdot V|_{t=0}&=\frac{u_t^2}{|\nabla u|^2}\nu^T\cdot D^2u\cdot \nu =\frac{u_t^2}{|\nabla u|^4}\nabla u ^T\cdot D^2u\cdot\nabla u\\
        &=\frac{u_t^2}{|\nabla u|^4}\left(\Delta u|\nabla u|^2-\mathrm{div}\left(\frac{\nabla u}{|\nabla u|} \right)|\nabla u|^3 \right)\\
        &=\frac{u_t^2}{|\nabla u|^2}\Delta u-\frac{u_t^2}{|\nabla u|}H,
    \end{split}
\end{equation*}
where $H$ is the non-normalized mean curvature, that is
\begin{equation*}
    H=(n-1)\sigma_1,
\end{equation*}
and $\sigma_1$ is defined in \eqref{meancurv}.

Then, since $u$ is given by \eqref{Sol on ball}, we have that $|\nabla u|^n=(x_1^2+\dots+x_n^2)^\frac{n}{2}=1$ on $\partial B_1$ and the first derivative of $j(t)$, evaluated at $t=0$, is 
\begin{equation*}
        j'(0)=\frac{\omega_n^{n-1}}{(n+2)^{n-1}}\int_{\partial B_1}|\nabla u|^{n+1}V|_{t=0}\cdot \nu\,d\mathcal{H}^{n-1}=\frac{\omega_n^{n-1}}{(n+2)^{n-1}}\int_{\partial B_1}V|_{t=0}\cdot\nu\,d\mathcal{H}^{n-1}.
\end{equation*}

Furthermore the second derivative, at the time $t=0$, is given by
\begin{equation*}
    \begin{split}
        j''(0)&=\frac{n-1}{n}\omega_n^{n-2}\left(\int_{\partial B_1}u_t\,d\mathcal{H}^{n-1} \right)^2+\frac{\omega_n^{n-1}}{(n+2)^{n-1}}\\
        &\times\left\{\int_{\partial B_1}\left[u_t^2(\Delta u-H)+V_t|_{t=0}\cdot \nu-2u_t\frac{\partial u_t}{\partial \nu}\right]\,d\mathcal{H}^{n-1}+\int_{B_1}(S_n^{ij}(D^2u))_tu_{tij}u\,dx +n\int_{\partial B_1}(V|_{t=0}\cdot \nu)^2\,d\mathcal{H}^{n-1} \right\}.
    \end{split}
\end{equation*}
Now, recalling that $\Delta u=n$ on $\partial B_1$ and $H|_{\partial B_1}=n-1$, we have the thesis.
\end{proof}
Once we know how to write the second variation of the torsion functional, we want to write in terms of spherical harmonics the deficit $\delta T$. This is the content of the following theorem.
\begin{teorema}
    In the same assumptions of Theorem \ref{3.3} consider the torsion deficit $\delta T(t)=1-T(\Omega(t))/T(B_1)$. Then
    \begin{equation*}
        \delta T(t)=-\frac{t^2}{2}\frac{n+2}{\omega_n}\left[(n+1)\sum_{k\geq 0}a_k^2-2\sum_{k\geq 0}ka_k^2+\int_{\partial B_1}A(\xi)\,d\mathcal{H}_\xi^{n-1}+\frac{1}{n}\sum_{k\geq 0}a_k^2(k^2-k)\right]+o(t^2),
    \end{equation*}
    where 
    \begin{equation*}
        a_k=\int_{\partial B_1}V(\xi)Y_k(\xi)\,d\mathcal{H}^{n-1}_\xi,
    \end{equation*}
    and $\{Y_k(\xi)\}_{k\in\mathbb{N}}$ is a family of \emph{spherical harmonics} that satisfy
    \begin{equation*}
        -\Delta_\xi Y_k=k(k+n-2)Y_k, \quad \|Y_k\|_{L^2}=1.
    \end{equation*}
\end{teorema}
\begin{proof}
We start by setting
\begin{equation*}
    V(x)=V|_{t=0}\cdot \nu(x),\quad A(x)=V_t|_{t=0}\cdot \nu(x).
\end{equation*}
We also introduce for $x\equiv (r,\xi)\in\Omega$
\begin{equation*}
    v(r,\xi)=u_t(x),
\end{equation*}
where $u_t$ is given by \eqref{omega}.
Now, from \eqref{der sec}, we have that
\begin{equation*}
    j'(0)=\frac{\omega_n^{n-1}}{(n+2)^{n-1}}\int_{\partial B_1}V(\xi)\,d\mathcal{H}_\xi^{n-1}=0.
\end{equation*}
This outcome is expected, as
$T$ reaches its maximum at $t=0$. With this notation, we also have
\begin{equation*}
    \begin{split}
        j''(0)&=\frac{n-1}{n}\omega_n^{n-2}\left(\int_{\partial B_1}V(\xi)\,d\mathcal{H}^{n-1}_\xi \right)^2+\frac{\omega_n^{n-1}}{(n+2)^{n-1}}\int_{\partial B_1}\left[(n+1)V^2(\xi)+A(\xi)+2V(\xi)\frac{\partial v}{\partial r}\bigg|_{r=1} \right]\,d\mathcal{H}^{n-1}_\xi\\
        &+\frac{\omega_n^{n-1}}{(n+2)^{n-1}}\int_{B_1}(S_n^{ij}(D^2u))_tu_{tij}u\,dx\\
        &=\frac{\omega_n^{n-1}}{(n+2)^{n-1}}\int_{\partial B_1}\left[(n+1)V^2(\xi)+A(\xi)+2V(\xi)\frac{\partial v}{\partial r}\bigg|_{r=1} \right]\,d\mathcal{H}^{n-1}_\xi+\frac{\omega_n^{n-1}}{(n+2)^{n-1}}\int_{B_1}(S_n^{ij}(D^2u))_tu_{tij}u\,dx.
    \end{split}
\end{equation*}

Now, for $t$ sufficiently small, $\partial\Omega(t)$ can be parametrized as
\begin{equation*}
    r(\xi,t)=1+V(\xi)t+\frac{A(\xi)}{2}t^2+o(t^2),\quad r\in\mathbb{R}^+,\quad \xi\in \partial B_1.
\end{equation*}

If we apply the Taylor expansion around $t=0$ we have
\begin{equation*}
    T(\Omega(t))=T(B_1)+t\underbrace{j'(0)}_{=0}+\frac{t^2}{2}j''(0)+o(t^2)
\end{equation*}
and then
\begin{equation*}
    \delta T(t)=1-\frac{T(\Omega(t))}{T(\Omega(0))}=-\frac{t^2}{2}\frac{j''(0)}{T(0)}+o(t^2).
\end{equation*}

Now we need an expression for $v(r,\xi)$. We observe that the equation for $u_t$ is $S_n^{ij}(D^2u)u_{tij}=0$ and, at $t=0$, since $S_n^{ij}(D^2u)=\delta_{ij}$, becomes $\Delta u_t=0$. Then in spherical coordinates we have
\begin{equation*}
    \begin{cases}
        -\frac{1}{r^{n-1}}\frac{\partial}{\partial r}\left(r^{n-1}\frac{\partial v}{\partial r}\right)-\frac{1}{r^2}\Delta_\xi v=0 & (r,\xi)\in(0,1)\times \partial B_1\\
        v(1,\xi)=-V(\xi) & \xi\in \partial B_1.
    \end{cases}
\end{equation*}
Now we want to write $V(\xi)$ in terms of a family of spherical harmonics $\{Y_k\}_{k\in\mathbb{N}}$, which satisfy for all $k\geq0$ 
\begin{equation}\label{propHarm}
    \begin{split}
        &V(\xi)=\sum_{k=0}^{+\infty} a_k Y_k(\xi),\quad \xi\in \partial B_1, \quad -\Delta_\xi Y_k=k(k+n-2)Y_k,\\
        &\|Y_k\|_{L^2}=1,\quad a_k=\int_{\partial B_1}V(\xi)Y_k(\xi)\,d\mathcal{H}^{n-1}_\xi,\quad a_0=0,\quad \|V\|_{L^2}^2=\sum_{k=0}^{+\infty} a_k^2.
    \end{split}
\end{equation}
Then, with the separation of variables
\begin{equation}\label{varsep}
    v(r,\xi)=\sum_{k=0}^{+\infty}R_k(r)Y_k(\xi),
\end{equation}
the equation becomes
\begin{equation*}
    -\frac{n-1}{r}\frac{\partial v}{\partial r}-\frac{\partial^2 v}{\partial r^2}-\frac{1}{r^2}\Delta_\xi v=0,
\end{equation*}
that is from \eqref{varsep}
\begin{equation*}
    \sum_{k=0}^{+\infty} \left[-\frac{n-1}{r}R'_k(r)Y_k(\xi)-R''_k(r)Y_k(\xi)-\frac{1}{r^2}R_k(r)\Delta_\xi Y_k(\xi) \right]=0.
\end{equation*}
Using \eqref{propHarm} we get
\begin{equation*}
    \sum_{k=0}^{+\infty} \left[-\frac{n-1}{r}R'_k(r)-R''_k(r)+\frac{k(k+n-2)}{r^2}R_k(r) \right]Y_k(\xi)=0,
\end{equation*}
but, since $\{Y_k\}_k$ is a basis of $L^2$, then we have that, for all $k\geq 0$,
\begin{equation*}
    \begin{cases}
        &R''_k(r)+\frac{n-1}{r}R'_k(r)-\frac{k(k+n-2)}{r^2}R_k(r)=0,\\
        &R_k(1)=-a_k,\\
        &R'_k(0)=0.
    \end{cases}
\end{equation*}
Explicitly solving the ODE, we arrive to
\begin{equation*}
    R_k(r)=c_1r^k+c_2r^{2-(n+k)},
\end{equation*}
with $c_1$, $c_2$ to be determined. Since we are interested in solutions that are $L^2$, we choose $c_2=0$ and then, from $R_k(1)=-a_k$ we find
\begin{equation*}
    R_k(r)=-a_kr^k.
\end{equation*}
Now we are interested in the radial derivative of $v(r,\xi)$, evaluated in $r=1$, i.e.,
\begin{equation*}
    \frac{\partial v}{\partial r}(r,\xi)\bigg|_{r=1}=\sum_{k=0}^{+\infty} R'_k(1)Y_k(\xi)=-\sum_{k=0}^{+\infty}ka_kY_k(\xi),
\end{equation*}
and, rewriting the torsion deficit, we have
\begin{equation*}
    \delta T(t)=-\frac{t^2}{2}\frac{n+2}{\omega_n}\left[\int_{\partial B_1}[(n+1)V^2(\xi)+A(\xi)-2V(\xi)\sum_{k\geq 0}ka_k Y_k(\xi)]\,d\mathcal{H}_\xi^{n-1}+\int_{B_1}(S_n^{ij}(D^2u))_tu_{tij}u\,dx\right]+o(t^2).
\end{equation*}
Recalling that
\begin{equation*}                   \|V\|_{L^2}^2=\sum_{k\geq 0}      a_k^2,\quad \int_{S^{n-     1}}V(\xi)Y_k(\xi)d\mathcal{H}_\xi^{n-1}=a_k,
\end{equation*}
and (see Appendix \ref{Appendix A})
\begin{equation}
    \label{AppendixA}
    \int_{B_1}(S_n^{ij}(D^2u))_tu_{tij}u\,dx=\frac{1}{n}\sum_{k\geq 0}a_k^2(k^2-k)
\end{equation}
we find
\begin{equation*}
    \delta T(t)=-\frac{t^2}{2}\frac{n+2}{\omega_n}\left[(n+1)\sum_{k\geq 0}a_k^2-2\sum_{k\geq 0}ka_k^2+\int_{\partial B_1}A(\xi)\,d\mathcal{H}_\xi^{n-1}+\frac{1}{n}\sum_{k\geq 0}a_k^2(k^2-k)\right]+o(t^2).
\end{equation*}
\end{proof}
Now we want to write, in terms of spherical harmonic functions, the \emph{Alexandrov-Fenchel deficit}.
\begin{lemma}
    The Alexandrov-Fenchel deficit is given by
    \begin{equation*}
        \delta_{AF}(t)=-\frac{t^2}{2}\frac{n+2}{\omega_n}\left[\int_{\partial B_1} A(\xi)\,d\mathcal{H}_\xi^{n-1}+(n-1)\sum_{k\ge 0}a^2_k \right].
    \end{equation*}
\end{lemma}
\begin{proof}
Since we are fixing the $(n-1)$-th quermassintegral, we can expand the volume inside the deficit and impose the conditions \eqref{der sec}
\begin{equation}\label{espDef}
    \begin{split}
        \delta_{AF}(t)&=1-\omega_n^{(n-1)(n+2)}\frac{W_0^{n+2}(\Omega(t))}{W_{n-1}^{n(n+2)}(\Omega(t))}\\
        &=1-\omega_n^{(n-1)(n+2)}\frac{W_0^{n+2}(\Omega(0))+t\frac{d}{dt}W_0^{n+2}(\Omega(t))|_{t=0}+\frac{t^2}{2}\frac{d^2}{dt^2}W_0^{n+2}(\Omega(t))|_{t=0}}{\omega_n^{n(n+2)}}+o(t^2).
    \end{split}
\end{equation}
In the same spirit of \eqref{espansione W_{n-1}}, we can state that
\begin{equation*}
    W_0(\Omega(t))=\omega_n+t\int_{\partial B_1}V(\xi)\,d\mathcal{H}_\xi^{n-1}+\frac{t^2}{2}\int_{\partial B_1}[A(\xi)+(n-1)V^2(\xi)]\,d\mathcal{H}_\xi^{n-1}+o(t^2),
\end{equation*}
where we observe that the first order term is zero. Then we have
\begin{equation*}
    \begin{split}
        W_0^{n+2}(\Omega(t))|_{t=0}&=\omega_n^{n+2},\\
        \frac{d}{dt}W_0^{n+2}(\Omega(t))|_{t=0}&=(n+2)W_0^{n+1}(\Omega(0))\frac{d}{dt}W_0(\Omega(t))|_{t=0}=0,\\
        \frac{d^2}{dt^2}W_0^{n+2}(\Omega(t))|_{t=0}&=(n+2)(n+1)W_0^n(\Omega(0))\left(\frac{d}{dt}W_0(\Omega(t))|_{t=0}\right)^2+(n+2)W_0^{n+1}(\Omega(0))\frac{d^2}{dt^2}W_0(\Omega(t))|_{t=0}\\
        &=(n+2)\omega_n^{n+1}\int_{\partial B_1}[A(\xi)+(n-1)V^2(\xi)]\,d\mathcal{H}_\xi^{n-1}.
    \end{split} 
\end{equation*}
Finally, from \eqref{espDef}, we can rewrite the deficit as
\begin{equation*}
    \begin{split}
        \delta_{AF}(t)&=-\frac{t^2}{2}\frac{n+2}{\omega_n}\int_{\partial B_1}[A(\xi)+(n-1)V^2(\xi)]\,d\mathcal{H}_\xi^{n-1}+o(t^2)\\
        &=-\frac{t^2}{2}\frac{n+2}{\omega_n}\left[\int_{\partial B_1} A(\xi)\,d\mathcal{H}_\xi^{n-1}+(n-1)\sum_{k\ge 0}a^2_k \right] .
    \end{split}
\end{equation*}
\end{proof}
From \eqref{der sec} we can rewrite the acceleration term as follows
\begin{equation}\label{newacc}
    \int_{\partial B_1}A(\xi)\,d\mathcal{H}_\xi^{n-1}=\frac{n^2-6n+7}{n-1}\int_{\partial B_1}|\nabla_\xi V|^2\,d\mathcal{H}^{n-1}_\xi-\frac{2(n-3)}{n-1}\int_{\partial B_1}S_2(D^2V)\,d\mathcal{H}^{n-1}_\xi.
\end{equation}
Our aim is to write the previous integral in terms of spherical harmonics.
\begin{lemma}
    The acceleration term is given by   \begin{equation}\label{accellerazione}
        \int_{\partial B_1}A(\xi)\,d\mathcal{H}_{\xi}^{n-1}=-\sum_{k\geq 0}a_k^2k(k+n-2).
    \end{equation}
\end{lemma}

\begin{proof}
Let's start by recalling that 
\begin{equation*}
    \int_{\partial B_1}Y_i(\xi)Y_j(\xi)\,d\mathcal{H}_\xi^{n-1}=0,\quad \text{ if } i\not=j,
\end{equation*}
we find
\begin{equation*}
    \begin{split}
        \int_{\partial B_1}|\nabla_\xi V|^2\,d\mathcal{H}_\xi^{n-1}&=-\int_{\partial B_1}V\Delta_\xi V\,d\mathcal{H}^{n-1}_\xi=-\int_{\partial B_1}\left(\sum_{k\geq 0}a_kY_k \right)\left(-\sum_{k\geq 0}a_k k(k+n-2)Y_k \right)\,d\mathcal{H}_\xi^{n-1}\\
        &=\sum_{k\geq 0}a_k^2 k(k+n-2)\int_{\partial B_1}Y_k^2\,d\mathcal{H}_\xi^{n-1}=\sum_{k\geq 0}a_k^2 k(k+n-2).
    \end{split}
\end{equation*}
For $S_2(D^2V)$, we recall the following equation (see \cite{VanBlargan2022QuantitativeQI}, formula 92) that holds true for $t$ sufficiently small and for every function that lives on the sphere
\begin{equation}\label{nueva}
    \int_{\partial B_1}S_k(D^2V)\,d\mathcal{H}^{n-1}=\frac{n-k}{k}\int_{\partial B_1}V^iV_j[T_{k-2}]_i^j(D^2V)\,d\mathcal{H}^{n-1},
\end{equation}
for $k\ge 2$.
From \eqref{nueva} we have
\begin{equation*}
    \int_{\partial B_1} S_2(D^2V)\,d\mathcal{H}^{n-1}=\frac{n-2}{2}\int_{\partial B_1}V^iV_j\delta_i^j\,d\mathcal{H}^{n-1}=\frac{n-2}{2}\int_{\partial B_1} |\nabla_\xi V|^2\,d\mathcal{H}^{n-1}=\frac{n-2}{2}\sum_{k\geq 0}a_k^2 k(k+n-2).
\end{equation*}
Then, from \eqref{newacc}, we get
\begin{equation*}
    \int_{\partial B_1}A(\xi)\,d\mathcal{H}_{\xi}^{n-1}=-\sum_{k\geq 0}a_k^2k(k+n-2).
\end{equation*}
\end{proof}
\section{Proof of Theorem \ref{MainTheo}}\label{Sezione 4}
Now we are able to prove the main theorem. Thanks to \eqref{accellerazione} we can rewrite the ratio $\delta T/\delta_{AF}$ as follows
\begin{equation*}
    \frac{\delta T(t)}{\delta_{AF}(t)}=\frac{(n+1)\sum_{k\ge 0}a^2_k-2\sum_{k\ge 0}ka^2_k -\sum_{k\geq 0}a_k^2k(k+n-2)+\frac{1}{n}\sum_{k\geq 0}a_k^2(k^2-k)}{(n-1)\sum_{k\ge 0}a_k^2-\sum_{k\geq 0}a_k^2k(k+n-2)}+o(1).
\end{equation*}

Then, since $a_0=0$, we have
\begin{equation}\label{liminf}
    \liminf_{t\rightarrow 0} \frac{\delta T(t)}{\delta_{AF}(t)}\geq \inf_{k\geq 2}\left\{\frac{n+1-2k-k(k+n-2)+\frac{k^2-k}{n}}{n-1-k(k+n-2)}\right\}=:\inf_{k\geq 2} f(k),
\end{equation}
where $a_1$ provides no contribution, indeed $k=1$ corresponds to a translation of $\Omega(0)$.

Now we want to show that the previous infimum is attained as $k\rightarrow+\infty$. In order to do this we see $f$ as the restriction to $\mathbb{N}_{2}$ of the following function
\begin{equation}\label{f tilde}
    \tilde f:x\in\mathbb{R}_{2}\mapsto \frac{n+1-2x-x(x+n-2)+ \frac{1}{n}(x^2-x)}{n-1-x(x+n-2)},
\end{equation}
where with $\mathbb{N}_2$ ($\mathbb{R}_2$) we denote the set of natural (real) numbers greater or equal then $2$. 
The strategy of the proof is the following. We will show first that if $\overline{x}\geq 2$ is a critical point for $\tilde f$, then it must be a \emph{local maximum}. From this it will be clear that 
\begin{equation}\label{infimo}
    \inf_{x\geq 2} \tilde f(x)=\min\{\tilde f(2), \, \lim_{x\rightarrow +\infty }\tilde f(x) \}.
\end{equation}
Once we know that \eqref{infimo} holds, we show that this minimum is the limit of $\tilde f$ as $x\rightarrow+\infty$. These are the contents of the following lemmas.
\begin{lemma}
    Let $\overline{x}\geq 2$. If $\tilde f'(\overline{x})=0$, then $\tilde f''(\overline{x})<0$, i.e., $\overline{x}$ is a local maximum for $\tilde f$.
\end{lemma}
\begin{proof}
    We call $N(x)$, $D(x)$ the numerator and denominator, respectively, in \eqref{f tilde}. Let $\overline{x}$ be a critical point of $\tilde f$, that is
    \begin{equation*}
        \tilde f'(\overline{x})=0\quad\text{if and only if} \quad N'(\overline{x})D(\overline{x})-N(\overline{x})D'(\overline{x})=0.
    \end{equation*}
    Then 
    \begin{equation*}
        \tilde f''(\overline{x})<0\quad\text{if and only if} \quad N''(\overline{x})D(\overline{x})-N(\overline{x})D''(\overline{x})<0.
    \end{equation*}
   We have 
    \begin{equation*}
        N''(\overline{x})D(\overline{x})-N(\overline{x})D''(\overline{x})=\frac{n-1}{n}D(\overline{x})-N(\overline{x}).
    \end{equation*}
    After straightforward computation we find
    \begin{equation}\label{condDerSec}
        \tilde f''(\overline{x})<0\quad\text{if and only if} \quad \overline{x}>\frac{n-1}{3n-1}.
    \end{equation}
    Since $\frac{n-1}{3n-1}<1$ and $\overline{x}\geq2$, the claim follows.
\end{proof}
\begin{lemma}\label{lemmainf}
    Let $\tilde f$ be defined as in \eqref{f tilde}, then
    \begin{equation*}
        \inf_{x\geq2}\tilde f(x)=\lim_{x\rightarrow+\infty}\tilde f(x).
    \end{equation*}
\end{lemma}
\begin{proof}
    It follows from the previous lemma that
    \begin{equation*}
        \inf_{x\geq 2} \tilde f(x)=\min\{\tilde f(2), \, \lim_{x\rightarrow +\infty }\tilde f(x) \}.
    \end{equation*}
    It remains to show that
    \begin{equation*}
        \frac{n^2+3n-2}{n(n+1)} =\tilde f(2)>\lim_{x\rightarrow +\infty }\tilde f(x)=\frac{n-1}{n}.
    \end{equation*}
    The previous inequality  is true if and only if
    \begin{equation*}
        3n-1>0,
    \end{equation*}
    which is true for all $n\geq 2$.
\end{proof}
Now we are able to prove Theorem \ref{MainTheo}.
\begin{proof}[proof of Theorem \ref{MainTheo} ]
    From \eqref{liminf} and Lemma \ref{lemmainf} follows that
    \begin{equation*}
        \liminf_{t\rightarrow 0} \frac{\delta T(t)}{\delta_{AF}(t)}\geq \inf_{k\geq 2} f(k)=\frac{n-1}{n}.
    \end{equation*}
    For the upper bound, let $\Omega$ be an open, bounded and convex set. From \eqref{min tor}-\eqref{max tor}, we find 
\begin{equation*}
    \begin{split}
        T(\Omega^*_{n-1})-T(\Omega)&\leq \frac{\omega_n^n\zeta_{n-1}^{n(n+2)}(\Omega)}{(n+2)^n}-\frac{|\Omega|^{n+2}}{(n+2)^n\omega_n^2}=\frac{W_{n-1}^{n(n+2)}(\Omega)}{(n+2)^n\omega_n^{n(n+1)}}-\frac{W_0^{n+2}(\Omega)}{(n+2)^n\omega_n^2}\\
        &=\frac{1}{(n+2)^n\omega_n^{n(n+1)}}\bigg[W_{n-1}^{n(n+2)}(\Omega)-\omega_n^{(n-1)(n+2)}W_0^{n+2}(\Omega) \bigg].
    \end{split}
\end{equation*}
After normalization, we have
\begin{equation*}
    \delta T=\frac{T(\Omega^*_{n-1})-T(\Omega)}{T(\Omega^*_{n-1})}\leq 1-\omega_n^{(n-1)(n+2)}\frac{W_0^{n+2}(\Omega)}{W_{n-1}^{n(n+2)}(\Omega)}=\delta_{AF}. 
\end{equation*}
\end{proof}
\appendix
\section[Proof of]{Proof of \eqref{AppendixA}}
\label{Appendix A}
We start by recalling the identity \eqref{cofproof}
\begin{equation*}
    S_k^{ij}(A)=\frac{1}{k}[T_{k-1}]_i^j(A),
\end{equation*}
that in our case is
\begin{equation*}
    S_n^{ij}(D^2u)=\frac{1}{n}[T_{n-1}]_i^j(D^2u)=\frac{1}{n(n-1)!}\delta_{jj_1\dots j_{n-1}}^{ii_1\dots i_{n-1}}u_{i_1j_1}\dots u_{i_{n-1}j_{n-1}}.
\end{equation*}
Then
\begin{equation*}
    \begin{split}
        (S_n^{ij}(D^2u))_tu_{ijt}&=\frac{1}{n}\left([T_{n-1}]_i^j(D^2u)\right)_tu_{ijt}=\frac{1}{n(n-1)!}\delta_{jj_1\dots j_{n-1}}^{ii_1\dots i_{n-1}}(u_{i_1j_1}\dots u_{i_{n-1}j_{n-1}})_tu_{ijt}\\
        &=\frac{n-1}{n(n-1)!}\delta_{jj_1\dots j_{n-1}}^{ii_1\dots i_{n-1}}(u_{i_1j_1})_tu_{i_{2}j_{2}}\dots u_{i_{n-1}j_{n-1}}u_{ijt}.
    \end{split}
\end{equation*}
We remark that on the ball $\Omega(0)=B_1$ the solution $u$ to \eqref{probtor} satisfies $u_{ij}=\delta_{ij}$. Then, on $B_1$, iterating \eqref{contr2}, we have
\begin{equation*}
    (S_n^{ij}(D^2u))_tu_{ijt}=\frac{1}{n}\delta_{ii_1}^{jj_1}(u_{i_1j_1})_t(u_{ij})_t=\frac{1}{n}\left[(\Delta u_t)^2-\sum_{ij}u_{ijt}^2 \right].
\end{equation*}
Since $u_t$ is harmonic in $B_1$, applying the divergence Theorem, we find
\begin{equation*}
    \int_{B_1}(S_n^{ij}(D^2u))_tu_{tij}u\,dx=-\frac{1}{n}\sum_{i,j}\int_{B_1}u^2_{tij}u\,dx=\frac{1}{n}\int_{B_1}u_{tij}u_ju_{ti}\,dx=\frac{1}{2n}\int_{B_1} (u_{ti}^2)_j u_j\,dx.
\end{equation*}
Now applying again the divergence Theorem, since on $\partial B_1$ it holds $|\nabla u|=1$ and $\Delta u=n$, 
\begin{equation}\label{Snesp}
    \int_{B_1}(S_n^{ij}(D^2u))_tu_{tij}u\,dx=\frac{1}{2n}\sum_i\int_{\partial B_1}u_{ti}^2\,d\mathcal{H}^{n-1}-\frac{1}{2}\sum_i\int_{B_1}u^2_{ti}\,dx.
\end{equation}
The second term in the last equality, integrating by parts, is
\begin{equation}\label{secter}
    \sum_i\int_{B_1}u_{ti}^2\,dx=\int_{\partial B_1}u_t\frac{\partial u_t}{\partial\nu}\,d\mathcal{H}^{n-1}=-\int_{\partial B_1}\frac{\partial u_t}{\partial\nu}V\cdot\nu\,d\mathcal{H}^{n-1}=-\int_{\partial B_1}\frac{\partial v}{\partial r}(r,\xi)\bigg|_{r=1}V(\xi)\,d\mathcal{H}^{n-1}_\xi=\sum_{k\geq 0}a_k^2k,
\end{equation}
where, again, we have used $\Delta u_t=0$.
For the first term in \eqref{Snesp} we recall the following integration by parts formula (see \cite[Theorem 5.4.13]{HP})
\begin{equation}\label{intXpar}
    \int_{\partial\Omega}\nabla f\cdot \nabla g\,d\mathcal{H}^{n-1}=-\int_{\partial\Omega}f\Delta g\,d\mathcal{H}^{n-1}+\int_{\partial\Omega}\left(\frac{\partial f}{\partial\nu} \frac{\partial g}{\partial\nu}+f\frac{\partial^2 g}{\partial\nu^2}+Hf\frac{\partial g}{\partial\nu}\right)\,d\mathcal{H}^{n-1}.
\end{equation}
Then, from \eqref{intXpar}, recalling that $u_t$ is harmonic, we have
\begin{equation}\label{priter}
    \begin{split}
         \int_{\partial B_1}\nabla & u_t\cdot \nabla u_t\,d\mathcal{H}^{n-1}=\int_{\partial B_1}\left(\frac{\partial u_t}{\partial\nu}\right)^2\,d\mathcal{H}^{n-1}+\int_{\partial B_1}u_t\frac{\partial^2 u_t}{\partial\nu^2}\,d\mathcal{H}^{n-1}+\int_{\partial B_1}Hu_t\frac{\partial u_t}{\partial\nu}\,d\mathcal{H}^{n-1}\\
         &=\int_{\partial B_1}\left(\frac{\partial v}{\partial r}(r,\xi)\bigg|_{r=1}\right)^2\,d\mathcal{H}^{n-1}_\xi-\int_{\partial B_1}V\cdot\nu \frac{\partial^2 v}{\partial r^2}(r,\xi)\bigg|_{r=1}\,d\mathcal{H}^{n-1}_\xi-(n-1)\int_{\partial B_1}V\cdot\nu \frac{\partial v}{\partial r}(r,\xi)\bigg|_{r=1}\,d\mathcal{H}^{n-1}_\xi\\
         &=\sum_{k\geq0}a_k^2[k^2+k(k-1)+k(n-1)].
    \end{split}
\end{equation}
then, from \eqref{Snesp}- \eqref{secter}-\eqref{priter}, we have
\begin{equation*}
    \int_{B_1}(S_n^{ij}(D^2u))_tu_{tij}u=\frac{1}{n}\sum_{k\geq 0}a_k^2(k^2-k),
\end{equation*}
that is \eqref{AppendixA}.

\subsection*{Acknowledgements}
We would like to thank Dr. Paolo Acampora and Dr. Emanuele Cristoforoni for their valuable advices that helped us to achieve these results.
\subsection*{Declarations}
\subsection*{Funding}
The authors were partially supported by Gruppo Nazionale per l’Analisi Matematica, la Probabilità e le loro Applicazioni
(GNAMPA) of Istituto Nazionale di Alta Matematica (INdAM).   
 \subsection*{Conflict of interst} We declare that we have no financial and personal relationship with other people or organizations.

\bibliographystyle{plain}
\bibliography{biblio}

@article{VanBlargan2022QuantitativeQI,
  title={Quantitative Quermassintegral Inequalities for Nearly Spherical Sets},
  author={C. VanBlargan and Y. Wang},
  journal={Communications in Contemporary Mathematics},
  year={2022}
}

@article{Nitsch2012AnIR,
  title={An isoperimetric result for the fundamental frequency via domain derivative},
  author={C. Nitsch},
  journal={Calculus of Variations and Partial Differential Equations},
  year={2012},
  volume={49},
  pages={323-335}
}

@book {AH,
    AUTHOR = {Atkinson, K. and Han, W.},
     TITLE = {Spherical harmonics and approximations on the unit sphere: an
              introduction},
    SERIES = {Lecture Notes in Mathematics},
    VOLUME = {2044},
 PUBLISHER = {Springer, Heidelberg},
      YEAR = {2012},
     PAGES = {x+244},
      ISBN = {978-3-642-25982-1},
   MRCLASS = {41-02 (33C55 41A30 41A63 42A10)},
  MRNUMBER = {2934227},
MRREVIEWER = {Feng\ Dai},
       DOI = {10.1007/978-3-642-25983-8},
       URL = {https://doi.org/10.1007/978-3-642-25983-8},
}

@article {MS,
    AUTHOR = {Masiello, A. L. and Salerno, F.},
     TITLE = {A quantitative result for the {$k$}-{H}essian equation},
   JOURNAL = {Nonlinear Anal.},
  FJOURNAL = {Nonlinear Analysis. Theory, Methods \& Applications. An
              International Multidisciplinary Journal},
    VOLUME = {255},
      YEAR = {2025},
     PAGES = {Paper No. 113776, 16},
      ISSN = {0362-546X,1873-5215},
   MRCLASS = {52A39 (35B35 35J60 35J96)},
  MRNUMBER = {4867241},
       DOI = {10.1016/j.na.2025.113776},
       URL = {https://doi.org/10.1016/j.na.2025.113776},
}

@article{Brandolini2009NewIE,
  title={New isoperimetric estimates for solutions to Monge―Amp{\`e}re equations},
  author={B. Brandolini and C. Nitsch and C. Trombetti},
  journal={Annales De L Institut Henri Poincare-analyse Non Lineaire},
  year={2009},
  volume={26},
  pages={1265-1275}
}

@article {tso,
    AUTHOR = {Tso, K.},
     TITLE = {On symmetrization and {H}essian equations},
   JOURNAL = {J. Analyse Math.},
  FJOURNAL = {Journal d'Analyse Math\'{e}matique},
    VOLUME = {52},
      YEAR = {1989},
     PAGES = {94--106},
      ISSN = {0021-7670,1565-8538},
   MRCLASS = {52A22 (52A40 58G20)},
  MRNUMBER = {981497},
MRREVIEWER = {Eric\ Grinberg},
       DOI = {10.1007/BF02820473},
       URL = {https://doi.org/10.1007/BF02820473},
}

@book{Schneider,
    author = { Schneider, R. },
    title = { Convex bodies : the Brunn-Minkowski theory },
    publisher = { Cambridge University Press Cambridge ; New York },
    pages = { xiii, 490 p. : },
    year = { 1993 },
    type = { Book },
    url = { http://www.loc.gov/catdir/toc/cam027/92011481.html http://www.loc.gov/catdir/samples/cam034/92011481.html },
    language = { English },
    subjects = { Convex bodies.; Analytic geometry }
}

@book {HP,
    AUTHOR = {Henrot, A. and Pierre, M.},
     TITLE = {Shape variation and optimization},
    SERIES = {EMS Tracts in Mathematics},
    VOLUME = {28},
      NOTE = {A geometrical analysis,
              English version of the French publication [MR2512810] with
              additions and updates},
 PUBLISHER = {European Mathematical Society (EMS), Z\"urich},
      YEAR = {2018},
     PAGES = {xi+365},
      ISBN = {978-3-03719-178-1},
   MRCLASS = {49Q10 (31B15 35J20 35R35 49-02 58E25 65J05)},
  MRNUMBER = {3791463},
MRREVIEWER = {Jan\ Soko\l owski},
       DOI = {10.4171/178},
       URL = {https://doi.org/10.4171/178},
}

@article {Talenti1981,
    AUTHOR = {Talenti, G.},
     TITLE = {Some estimates of solutions to {M}onge-{A}mp\`ere type
              equations in dimension two},
   JOURNAL = {Ann. Scuola Norm. Sup. Pisa Cl. Sci. (4)},
  FJOURNAL = {Annali della Scuola Normale Superiore di Pisa. Classe di
              Scienze. Serie IV},
    VOLUME = {8},
      YEAR = {1981},
    NUMBER = {2},
     PAGES = {183--230},
      ISSN = {0391-173X,2036-2145},
   MRCLASS = {53C45 (35J60)},
  MRNUMBER = {623935},
MRREVIEWER = {V.\ T.\ Fomenko},
       URL = {http://www.numdam.org/item?id=ASNSP_1981_4_8_2_183_0},
}

@article {BNT2,
    AUTHOR = {Brandolini, B. and Nitsch, C. and Trombetti, C.},
     TITLE = {Shape optimization for {M}onge-{A}mp\`ere equations via domain
              derivative},
   JOURNAL = {Discrete Contin. Dyn. Syst. Ser. S},
  FJOURNAL = {Discrete and Continuous Dynamical Systems. Series S},
    VOLUME = {4},
      YEAR = {2011},
    NUMBER = {4},
     PAGES = {825--831},
      ISSN = {1937-1632,1937-1179},
   MRCLASS = {49Q10 (35J96 52A41)},
  MRNUMBER = {2746444},
MRREVIEWER = {Luca\ Granieri},
       DOI = {10.3934/dcdss.2011.4.825},
       URL = {https://doi.org/10.3934/dcdss.2011.4.825},
}

@article {Wang,
    AUTHOR = {Wang, X. J.},
     TITLE = {A class of fully nonlinear elliptic equations and related
              functionals},
   JOURNAL = {Indiana Univ. Math. J.},
  FJOURNAL = {Indiana University Mathematics Journal},
    VOLUME = {43},
      YEAR = {1994},
    NUMBER = {1},
     PAGES = {25--54},
      ISSN = {0022-2518,1943-5258},
   MRCLASS = {35J65},
  MRNUMBER = {1275451},
MRREVIEWER = {John\ Urbas},
       DOI = {10.1512/iumj.1994.43.43002},
       URL = {https://doi.org/10.1512/iumj.1994.43.43002},
}

@article {Lions,
    AUTHOR = {Lions, P.-L.},
     TITLE = {Two remarks on {M}onge-{A}mp\`ere equations},
   JOURNAL = {Ann. Mat. Pura Appl. (4)},
  FJOURNAL = {Annali di Matematica Pura ed Applicata. Serie Quarta},
    VOLUME = {142},
      YEAR = {1985},
     PAGES = {263--275},
      ISSN = {0003-4622},
   MRCLASS = {58G25 (35B32 35J50 35J60 49G05)},
  MRNUMBER = {839040},
MRREVIEWER = {R\'emi\ Vaillancourt},
       DOI = {10.1007/BF01766596},
       URL = {https://doi.org/10.1007/BF01766596},
}

@book {HCG,
     TITLE = {Handbook of convex geometry. {V}ol. {A}, {B}},
    EDITOR = {Gruber, P. M. and Wills, J. M.},
 PUBLISHER = {North-Holland Publishing Co., Amsterdam},
      YEAR = {1993},
     PAGES = {Vol. A: lxvi+735 pp.; Vol. B: pp. i--lxvi and 737--1438},
      ISBN = {0-444-89598-1},
   MRCLASS = {52-06 (52-00)},
  MRNUMBER = {1242973},
}

@article {BT,
    AUTHOR = {Brandolini, B. and Trombetti, C.},
     TITLE = {Comparison results for {H}essian equations via symmetrization},
   JOURNAL = {J. Eur. Math. Soc. (JEMS)},
  FJOURNAL = {Journal of the European Mathematical Society (JEMS)},
    VOLUME = {9},
      YEAR = {2007},
    NUMBER = {3},
     PAGES = {561--575},
      ISSN = {1435-9855,1435-9863},
   MRCLASS = {35J60 (35B05)},
  MRNUMBER = {2314107},
MRREVIEWER = {Fabiana\ Leoni},
       DOI = {10.4171/JEMS/88},
       URL = {https://doi.org/10.4171/JEMS/88},
}

@article {DPGC,
    AUTHOR = {Della Pietra, F. and Gavitone, N. and Xia, C.},
     TITLE = {Symmetrization with respect to mixed volumes},
   JOURNAL = {Adv. Math.},
  FJOURNAL = {Advances in Mathematics},
    VOLUME = {388},
      YEAR = {2021},
     PAGES = {Paper No. 107887, 31},
      ISSN = {0001-8708,1090-2082},
   MRCLASS = {35A23 (35J25 52A39)},
  MRNUMBER = {4288211},
       DOI = {10.1016/j.aim.2021.107887},
       URL = {https://doi.org/10.1016/j.aim.2021.107887},
}

@article {L,
    AUTHOR = {Le, N. Q.},
     TITLE = {The eigenvalue problem for the {M}onge-{A}mp\`ere operator on
              general bounded convex domains},
   JOURNAL = {Ann. Sc. Norm. Super. Pisa Cl. Sci. (5)},
  FJOURNAL = {Annali della Scuola Normale Superiore di Pisa. Classe di
              Scienze. Serie V},
    VOLUME = {18},
      YEAR = {2018},
    NUMBER = {4},
     PAGES = {1519--1559},
      ISSN = {0391-173X,2036-2145},
   MRCLASS = {47A75 (35J70 35J96)},
  MRNUMBER = {3829755},
MRREVIEWER = {Mauricio\ Alexander\ Rivas},
}

@article {reilly73,
    AUTHOR = {Reilly, R. C. },
     TITLE = {On the {H}essian of a function and the curvatures of its graph},
   JOURNAL = {Michigan Math. J.},
    VOLUME = {20},
      YEAR = {1973},
     PAGES = {373--383},
}

@article {trudinger97,
    AUTHOR = {Trudinger, N. S. },
     TITLE = {On  new isoperimetric inequalities and symmetrization},
   JOURNAL = {J. Reine Angew. Math.},
    VOLUME = {488},
      YEAR = {1997},
     PAGES = {203--220},
}

@article {caffarelli,
    AUTHOR = {Caffarelli, L. and Nirenberg, L. and Spruck, J.},
     TITLE = {The {D}irichlet problem for nonlinear second-order elliptic
              equations. {III}. {F}unctions of the eigenvalues of the
              {H}essian},
   JOURNAL = {Acta Math.},
  FJOURNAL = {Acta Mathematica},
    VOLUME = {155},
      YEAR = {1985},
    NUMBER = {3-4},
     PAGES = {261--301},
      ISSN = {0001-5962,1871-2509},
   MRCLASS = {35J65 (53C40 58G30)},
  MRNUMBER = {806416},
MRREVIEWER = {Philippe\ Delano\"e},
       DOI = {10.1007/BF02392544},
       URL = {https://doi.org/10.1007/BF02392544},
}

\Addresses
\end{document}